\documentclass[11pt]{article}

\usepackage{graphicx}
\usepackage{color}
\usepackage[all]{xy}
\usepackage{amssymb}

\def\N {{\rm I\! N}}

\def\R {{\rm I\! R}}

\def\F {{\rm I\! F}}

\newcommand{\dst}{\displaystyle}
\newtheorem{defn}{Definition}
\newtheorem{lem}[defn]{Lema}
\newtheorem{prop}[defn]{Proposition}
\newtheorem{thm}[defn]{Theorem}
\newtheorem{cor}[defn]{Corollary}

\newtheorem{rem}[defn]{Remark}
\newtheorem{exmpl}[defn]{Example}

\def\dst{\displaystyle}

\def\square{\ifmmode\sqr\else{$\sqr$}\fi}
\def\sqr{\vcenter{
         \hrule height.1mm
         \hbox{\vrule width.1mm height2.2mm\kern2.18mm
\vrule width.1mm}
         \hrule height.1mm}}

\newenvironment{proof}[1]{
  \trivlist \item[\hskip \labelsep{\it #1}]}{\hfill\mbox{$\square$}
  \endtrivlist}
\begin{document}

\title{A new formalism for the study of Natural Tensor Fields of type (0,2) on Manifolds and Fibrations. }

\author{{Guillermo  Henry}\\Departamento de Matem\'atica, FCEyN,\\
 Universidad de Buenos Aires and CONICET, Argentina.}

\date{}

\maketitle


{\bf{Abstract.}} In order to study tensor fields of type (0,2) on  manifolds and fibrations we introduce a new formalism that we called \textit{s-space}. With the help of these objects we  generalized the concept of natural tensor without making use of the theory of natural operators and differential invariants.

\vspace{0.5 cm}


{\bf{Mathematics Subject Clasification (2000):\  }}  53C20 $\cdot$  53B21 $\cdot$  53A55 $\cdot$ 55R10

\section{Introduction.}

In \cite{Ko-Se}, Kowalski and Sekizawa defined and characterized the $natural\ tensor\ fields$ of type $(0,2)$ on the tangent bundle $TM$ of a manifold $M$. More precisely, let $\tilde{g}$ be a metric on $TM$ which cames from a second order natural transformation of a metric $g$ on $M$. Then  there are natural $F-metrics$ $\xi_1,\xi_2$ and $\xi_3$ (i.e. a bundle morphism of the form $\xi:TM\oplus TM\oplus TM\longrightarrow M\times \R$ linear in the second and in the third argument) derived from $g$, such that $\tilde{g}=\xi_1^{s,g}+ \xi_2^{h,g}+\xi_3^{v,g}$ with $\xi_1$ and $\xi_3$  symmetric , where $\xi_1^{s,g}, \xi_2^{h,g}$ and $\xi_3^{v,g}$ are the classical Sasaki, horizontal and vertical lift of $\xi_1,\xi_2$ and $\xi_3$ respectively. Also Kowalski and Sekizawa \cite{Ko-Se-2} study the $natural\ tensor\ fields$ on the linear frame bundles of a manifold endowed with a linear connection.

In \cite{MCC-GK},  Calvo and Keilhauer showed that given a Riemannian manifold $(M,g)$ any $(0,2)$ tensor field on $TM$ admits a global matrix representation. Using this one to one relationship, they defined and characterized what they called $natural\ tensor$. In the symmetric case this concept coincide with the one of Kowalski and Sekizawa.  Keilhauer \cite{GK} defined and characterized the tensor fields of type $(0,2)$ on the linear frame bundle of a Riemannian manifold endowed with a linear connection.  The $natural\ tensors$ on the tangent and cotangent bundle of a semi Riemannian manifold was characterized by Araujo and Keilhauer in \cite{A-GK}. The idea of all these works (\cite{A-GK},\cite{MCC-GK} and \cite{GK}) is to lifted to a suitable fiber bundle a tensor field on the tangent bundle,  cotangent bundle and  linear frame bundle respectively, so that to look at them as a global matricial maps. The principal difference with the works  \cite{Ko-Se} and \cite{Ko-Se-2} is that they do not make use of the theory of differential invariant developed by Krupka \cite{Krup}, (see also  \cite{Ko-Mi-Slovak} and \cite{Krup-Janys}).

The aim of this work is generalized the notion of natural tensor fields in the sense of \cite{A-GK},\cite{MCC-GK} and \cite{GK} to manifolds and fibrations. With this purpose we introduce the concept of \textit{s-space}. In Section $2$, we define and give some examples of \textit{s-spaces}. We also see general properties of \textit{s-spaces}, for example that there exist a one to one relationship between the tensor fields of type $(0,2)$ and some types of matricial maps. This  relationship   allows us to study  the tensor fields in the sense of \cite{MCC-GK}. We characterize the \textit{s-spaces} which its group acts without fixed point. We study some general statement of \textit{morphisms of s-spaces} and tensor fields on manifolds in  Section $3$. In Section $4$, we define \textit{connections} on \textit{s-spaces} (that agree with the well known notion of connection when the \textit{s-space} is also a principal fiber bundle). We give a condition that a \textit{s-space} endowed with a connection has to satisfies to has a parallelizable space manifold. Also, help by a $connection$  we show an useful way of lift metrics on the manifold to the space manifold of the \textit{s-space}.  The concept of \textit{s-space} gives several notions of naturality.   The $\lambda- natural$ and $\lambda-natural$ tensors with respect to a fibration are define in section $5$. We also give examples and we see that these notions extend that one of  \cite{A-GK},\cite{MCC-GK} and \cite{GK}. In Section $7$ we define the notion of \textit{atlas of s-spaces} and we use them to generalized the $\lambda-naturality$. In Section $8$, we consider some $s-spaces$ over a Lie group and characterized the $natural$ tensors fields on it. Finally, we study the bundle metrics on a principal fiber bundle endowed with a linear connection.


\section{s-spaces.}

\begin{defn}

Let $M$  be a  manifold of dimension $n$. A collection
$\lambda=(N,\psi,O,R,\{e_i\})$  is called a s-space over $M$ if:

\begin{itemize}

\item [a)] $N$ be a manifold.

\item [b)] $\psi:N\longrightarrow M $ is a submersion.

\item [c)] $O$ is a Lie group  and  $R$ is a right action of the group  $O$ over $N$ which is  transitive in each fibers. The action also satisfies that $\psi\circ R_a=\psi$ for all $a\in O$.

\item [d)]  $e_i:N\longrightarrow TM$, with  $1\leq i\leq n$,   are differential functions
 such that $\{e_1(z),\dots,e_n(z) \}$ is a base of $M_{\psi(z)}$ for all  $z\in N$.

\end{itemize}

\end{defn}


If $\psi(z)=p$, then $\{e_1(z),\dots, e_n(z)\}$
  and $\{e_1(z.a),\dots,e_n (z.a) \}$ are bases of  $M_p$. Therefore there exists an invertible matrix $L(z,a)$ such that
 $\{e_i(z.a)\}=\{e_i(z)\}.L(z,a)\ $, (i.e. $e_i(z.a)=\sum_{j=1}^ne^l(z)L^l_i(z,a)$ for $1\leq i \leq n$). If the matrix $L$ only depends of the parameter  of the Lie group $O$, we have a differentiable map $$ L:O\longrightarrow GL(n)\ \ \ \mbox{ such that} \ \ \ \{e_i\}\circ
 R_a=\{e_i\}.L(a)$$ that we called {\textit{the base change morphism of the s-space $\lambda$}}. It easy to see that $L$ is a group morphism. In this case we said that  $\lambda$ have a {\it{rigid base change}}. From now on, we will consider only this class of s-spaces.

In the sequel, unless otherwise stated, $\dim M=n$, $\dim O= k$ and we will denote the Lie algebra of $O$ by ${\mathfrak{o}}$. Also, we assume that all tensor are of type $(0,2)$.

\begin{exmpl}\label{ejsecanonico} Let $LM$ be the frame bundle of a manifold $M$. $LM$ induce a s-space $\lambda=(LM,\pi, GL(n),(\ \cdot\ ), \{\pi_i\})$ over $M$, where $\pi$ is the projection of the bundle, $(\ \cdot\ )$ is the natural action of the general linear group over $LM$ and $\pi_i(p,u)=u_i$. The base change morphism is $L(a)=a$ for all $a\in GL(n)$.  This example shows that every manifolds admits at least one s-space. For simplicity of notation, let us denote this s-space by $LM$ too.  If we consider a Riemannian metric on $M$ or an orientation, then the bundle of orthonormal frames and the bundle of orientated bases induced similar s-spaces over $M$.

\end{exmpl}

\begin{exmpl} Let $\alpha=(P,\pi, G,\ \cdot\ )$ be a principal fiber bundle over $M$, and $\omega$ be a connection on $\alpha$. Let $\lambda=(N,\psi,O,R,\{e_i\})$ where
\begin{itemize}
\item[a)] $N=\{(p,u,w): p\in P, u\ \mbox{is a base of }\ M_{\pi(p)}\ \mbox{and}\  w\ \mbox{is a base of }\ {\mathfrak{g}} \}$
\item [b)] $\psi(p,u,w)=p$.
\item [c)]$O=GL(n)\times GL(k)$ and  $R_{(a,b)}(p,u,w)=(p,u.a,w.b)$
\item [d)]  For $1\leq i \leq n$ and  $1\leq j\leq k$,  $e_i(p,u,w)$ is the horizontal lift with respect to $\omega$ of  $u_i$ at $p$
  and $e_{n+j}(p,u,w)$ is the only vertical vector on $P_p$ such that $\omega(p)(e_{n+j}(p,u,w))=w_j$.

\end{itemize}
$\lambda$ is a s-space over $P$ and it's base change morphism is given by $L(a,b)=\pmatrix{a
& 0 \cr 0 & b}$.
\end{exmpl}

\begin{exmpl}\label{ejseLM}
This example can be found in \cite{GK}. Let $M$ be a manifold and $\nabla$ be a linear connection on it. Let $K:TTM\longrightarrow TM$ be the connection function induced by $\nabla$ ( i.e. $K$ is the unique function that satisfies: for $v\in M_p$, $K\mid_{TM_v}:TM_v\longrightarrow M_p$ is a surjective linear map and  for any vector field $Y$ on $M$ such  that $Y(p)=v$, we have that $K(Y_{*_p}(w))=\nabla_wY$). For $1\leq i,j\leq n$, consider the 1-forms  $\theta^i$ and $\omega^i_j$ defined by $\pi_{*_{(p,u)}}(b)=\dst\sum_{i=1}^n\theta^i(p,u)(b)u_i$ and $K((\pi
_j)_{*_{(p,u)}}(b))=\dst\sum_{i=1}^n\omega^i_j(p,u)(b)u_i.$ Let $\lambda=(LM\times GL(n),\psi, GL(n), R,\{H_i,V^i_j\})$  where $\psi(p,u,b)=(p,u.b)$, the action is $R_a(p,u,b)=(p,u.a,a^{-1}b)$ and $\{H_i,V^i_j\}$ is  dual to $\{\theta^i,\omega^i_j\}$. $\lambda$ is  a s-space over the frame bundle of $M$ with base change morphism $L(a)\equiv Id_{n\times n}$.
\end{exmpl}

The importance of the s-spaces for the study of the tensors on  manifolds is given by the following proposition:

\begin{prop}\label{repmatricial} Let $\lambda=(N,\psi,O,R,\{e_i\})$ be a s-space over $M$ and $L$ be the  base change morphism of $\lambda$. There is a one to one correspondence between tensor fields of type $(0,2)$ on $M$  and differentiable maps $^{\lambda}T:N\longrightarrow \R^{n\times n}$ that satisfy the invariance property
$$^{\lambda}T\circ R_a=(L(a))^t.^{\lambda}T.L(a)\ $$
\end{prop}

\begin{proof}{{Proof.}}
Let $T$ be a tensor on $M$. Consider the matrix function $^{\lambda}T:N\longrightarrow \R^{n\times n}$ defined by
$[^{\lambda}T(z)]^i_j=T(\psi(z))(e_i(z),e_j(z))$. For $a\in O$, we have that the  $(i,j)$ entry of the matrix $^{\lambda}T(z.a)$ is
$[^{\lambda} T(z.a)]^i _j=T(\psi(z.a))(e_i(z.a),e_j(z.a))$
$=T(\psi(z))(\dst\sum_{r=1}^n e_r(z)L(a)^r_i,
\dst\sum_{s=1}^n e_s(z)L(a)^s_j )$
 $= \sum_{r,s=1}^n L(a)^r_i . ^{\lambda}T(z)^r_s.L(a)_j^s \ $, hence $^{\lambda}T$ satisfies the invariance property. Let $F:N\longrightarrow \R^{n\times n}$ be a differentiable function that satisfies the invariance property, we are going to show that there exists a unique tensor $T$ on $M$ such that $^{\lambda}T=F$. If $X$ is a vector field on $M$, then it induce a map $^{\lambda}X=(x_1,\dots,x_n):N\longrightarrow \R^n$ where $X(\psi(z))=\dst\sum_{i=1}^nx_i(z)e_i(z)$. It is easy to check that $^{\lambda}X\circ R_a= ^{\lambda}X. [L(a)^t]^{-1}$. Then, we define $T(p)(X,Y)=^{\lambda}X(z).F(z).(^{\lambda}Y(z))^t$ where $\psi(z)=p$.  Consider  $z$ and  $\bar{z}$ such that $\psi(z)=\psi(\bar{z})=p$. Since $O$ acts transitively on the fibers of $N$, there exists $a\in O$ that satisfies  $\bar{z}=z.a$. Therefore, $ ^{\lambda}X(\bar{z}).F(\bar{z}).(^{\lambda}Y(\bar{z}))^t= ^{\lambda}X(z).(L(a)^t)^{-1}.L(a)^t.^{\lambda}F(z).L(a).(L(a))^{-1}(^{\lambda}Y(z))^t=
  \ \ \ ^{\lambda}X(z).F(z).(^{\lambda}Y(z))^t$, what it  prove that $T$ it is well defined.  Given $X$ and $Y$ vector fields on $M$, $T(X,Y):M\longrightarrow \R$ is a differentiable function because  $T(X,Y)\circ \psi$ is differentiable and $\psi$ is a submersion. Since $T$ is $\cal F $$(M)$-bilinear, we conclude that $T$ is a tensor of type (0,2) on $M$. Finally, it is clear  that $^{\lambda}T=F$.

\end{proof}

\begin{thm}\label{teofp}

Let  $\lambda=(N,\psi,O,R,\{e_i\})$ be a s-space over $M$, such that $O$ acts without fixed point (i.e. if $z.a=a$ then $a=e$), then     $(N,\psi,O,R)$ its a principal fiber bundle over $M$.

\end{thm}

Let us denote by  $z\sim z'$ the  equivalence relation induced by  the action of the group $O$ on the manifold $N$.  To prove the previous Theorem we will need the following next two lemmas.

\begin{lem}\label{lemafp1}
Let $\lambda=(N,\psi,O,R,\{e_i\})$ be a s-space over $M$. Then $N/O$ has differentiable manifold structure and $\pi:N\longrightarrow N/O$ is a submersion.

\end{lem}

\begin{proof}{{Proof.}} Consider the map $\rho:N\times
N\longrightarrow M\times M$ defined by $\rho(z,z')=(\psi(z),\psi(z'))$.  $\rho$ is a submersion since $\psi$ it is.
Let the set $\bar{\Delta}=\{(z,z'): z\sim z'\}$ and $\Delta$ be  the diagonal submanifold of $N\times N$. Since $z\sim z'$ if and only if $\psi(z)=\psi(z')$, we have that $\bar{\Delta}=\rho^{-1}(\Delta)$. Therefore $\bar{\Delta}$ is a closed submanifold of $N\times N$. It is  well know (see for example \cite{D}) that if a group $O$ acts on a manifold $N$, $N/O$ has a structure of differentiable manifold such that the canonical  projection $\pi$ is a submersion if and only if $\bar{\Delta}$ is a closed submanifold of $N\times N$. In this case, the differentiable structure of $N/O$ is unique.

\end{proof}

\begin{lem}\label{lemafp2} Under the hypotheses of the previous lemma:
\begin{itemize}
\item [i)] $N/O$ is diffeomorphic to $M$.
\item [ii)] $\ker \pi_{*}=\ker \psi_*$.
\end{itemize}

\end{lem}

\begin{proof}{\bf{Proof.}}
Let $f:N/O\longrightarrow M$  defined by $f([z])=\psi(z)$. By definition $f\circ \pi=\psi$, then $f$ is differentiable and $\ker \pi_*\subseteq \ker
\psi_*$. In the other hand, let $g:M\longrightarrow N/O$, defined by $g(p)=\pi(z)$ where $z\in N$ satisfies that $\psi(z)=p$. Since $O$ acts transitively on the fibers of $N$, $g$ is well defined. As $\pi=g\circ \psi$ we have that $g$ is a differentiable function and that $\ker
\psi_* \subseteq \ker \pi_*$. An easy verification shows that   $g\circ f=Id_{N\slash O}$ and $f\circ g =Id_M$.

\end{proof}

\begin{rem}
If $\lambda=(N,\psi,O,R,\{e_i\})$ is a s-space over $M$, then $(N,\psi,O,R)$ is a principal fiber bundle over $N/O$.

\end{rem}

\begin{proof}{{Proof of Theorem \ref{teofp}.}} It remains to prove that  $(N,\psi,O,R)$
 satisfies the local triviality property, (i.e.  all $p\in M$  has an open neighbour  $U$ on $M$, and a diffeomorphism $\tau:\psi^{-1}(U)\longrightarrow U\times O$ such that $\tau=(\psi,\phi)$, where $\phi(z.a)=\phi(z).a$ for all
 $a\in O$). Let $p\in M$, take $[z_0]\in N/O$ such that $f([z_0])=p$. As  $(N,\psi,O,R)$ is a principal fiber bundle over $N/O$, there exist an open neighbour $V$ of $[z_0]$ and a diffeomorphism $\bar{\tau}=(\pi(z),\bar{\phi}(z))$ such that satisfy the local triviality property. $U=f(V)$ is an open neighbour of $p$ on $M$, since f is a diffeomorphism, and it satisfies that $\psi^{-1}(U)=\pi^{-1}(V)$ . Finally, if we define $\tau:\psi^{-1}(U)\longrightarrow U\times O$ by $\tau(z)=(\psi(z),\bar{\phi}(z))$, $U$ and $\tau$ satisfy the local triviality property on $p$.

\end{proof}

\begin{rem} Note  that  there exist  s-spaces that are not principal fiber bundles. For example, let $\lambda=(\ \R^n\times (\R^n-\{0\}),pr_1,GL(n),R,\{e_i\})$ over $R^n$, where $pr_1(p,q)=p$, $R_a(p,q)=(p,q.a)$ and $e_i(p,q)=\frac{\partial}{\partial u_i}|_p$ is the base of $\R^n_p$ induced by the canonical coordinate system of $\R^n$ .
\end{rem}

If we say that a s-space $\lambda=( N,\psi,O,R,\{e_i\})$ over $M$ is a principal fiber bundle, we want to say that $( N,\psi,O,R)$ is a principal fiber bundle over $M$.

We denote by $S_z=\{a\in O: z.a=z\}$ the stabilizer's group of the action $R$ at $z$. It is well know that,  if for a point $z\in N$ the orbit $z.O$ is locally closed (i.e. if $w\in z.O$, there exist an open neighbour $V$ of $w$ on $N$, such that $V\cap z.O$ is a closed  set of $V$ ), then $z.O$ is a submanifold of $N$ and $f_z([a])=z.a$ is a diffeomorphism between $O/S_z$ and $z.O$, see \cite{D}.

\begin{prop}
Let  $\lambda=(N,\psi,O,R,\{e_i\})$ be a s-space over $M$, then
\begin{itemize}
\item [i)] There exists $s\in \N_0$ such that $\dim S_z=s$ for all $z\in N$.

\item [ii)] $\dim N=\dim M + \dim O - s$.

\end{itemize}

\end{prop}

\begin{proof}{{Proof.}}

Let $z\in N$ and $\psi(z)=p$. That $\dim N= \dim \ker \psi_{*_z}+\dim M\ $ and $\dim \ker \psi_{*_z}= \dim \psi^{-1}(p)$ follow from the fact that $\psi$ is a submersion. Note that $z.O=\psi^{-1}(p)$, since $O$ acts transitively on the fibers. As $\psi^{-1}(p)$ is locally closed, we have  $\dim O\slash S_z =\dim \psi^{-1}(p)$. Therefore, $\dim N= \dim M+\dim O -\dim S_z $ for all $z$, so $\dim S_z\equiv s$ is constant, which completes the proof.

\end{proof}

Given a s-space $\lambda$  over $M$,  it will be very important to know the tensors on $M$ that satisfy that $^{\lambda}T$ is a constant matrix. It is clear that not for every matrix $A\in \R^{n\times n}$ there exists a tensor $T$ on $M$ such that $ ^{\lambda}T=A$. From proposition \ref{repmatricial},  we know that a necessary and sufficient condition for this happens  is that $L(a)^t.A.L(a)=A$ for all $a\in O$. In that case, we said that $\lambda$ \textit{admits matrix  representations of type A}. In the last part of the Section we show some conditions that a s-space has to satisfies to admits matrix representation of certain class of diagonal matrix.
For $\nu=0,1,\cdots,n-1$,  we denote by $I_{\nu}$ the following matrix of $\R^{n\times n}$
$$I_{\nu}=\pmatrix{-1& & & & & \cr & \ddots^{\nu} & & & & \cr & &
-1& & & \cr & & & 1 & & \cr & & & & \ddots^{n-\nu} & \cr & & & &
& 1}\ \mbox{if}\  \nu \geq 1\  \mbox{ and }\  I_0=Id_{n\times n}\ $$
With $O_{\nu}$ we denote the orthonormal group of index $\nu$. If $\nu=0$ then  $O_0=O(n)$.

\begin{prop}
Let $\lambda=(N,\psi,O,R,\{e_i\})$ be a s-space over $M$ with base change morphism $L$. If $0\leq \nu\leq n-1$, the following conditions are equivalent:
\begin{itemize}
\item[i)] $Img(L)\subseteq O_{\nu}$.

\item[ii)] $\lambda$ admits matrix representations of type $I_{\nu}$.
\item[iii)] There is a semi-Riemannian metric on $M$ of signature $\nu$ such that $\{e_1(z),\dots,e_n(z)\}$ is an orthonormal base of $M_{\psi(z)}$ for all $z\in N$.
\item[iv)] There exists a tensor $T$ on $M$  that satisfies $^{\lambda}T(z)=I_{\nu}$ for all $z\in \psi^{-1}(p_0)$ and for a $p_0\in M$.

\end{itemize}
\end{prop}
\begin{proof}{{Proof.}}$i)\Longrightarrow ii)$ Consider the constant map $F\equiv I_{\nu}$. Since $F$  satisfies the invariance property, it follows from the Proposition \ref{repmatricial} the existence of a tensor that satisfies $^{\lambda}T=I_{\nu}$.
 $ii) \Longrightarrow iii)$ If $^{\lambda}T=I_{\nu}$, then $T$ is a semi-Riemannian metric of index $\nu$ and $T(\psi(z))(e_i(z),e_j(z))=[I_{\nu}]^i_j$.
  $\ iii)\Longrightarrow iv)$ is immediately.
  $iv)\Longrightarrow i)$ Let $a\in O$ and $z_0$ such that $\psi(z_0)=p_0$, then $I_{\nu}=I_{\nu}(z_0.a)=L(a)^t.I_{\nu}.L(a)$ for all $a\in O$.

\end{proof}

The next Proposition is a consequence of the fact that $O(m)\cap O_{\nu}=\{D\in O(m): D=\pmatrix{A & 0 \cr 0 & B}\ con\
A\in O(\nu)\ y\ B\in O(m-\nu)\}.$
\begin{prop}\label{propInu1} Let $\lambda=(N,\psi,O,R,\{e_i\})$ be a s-space over $M$ with base change morphism $L$ and $1\leq \nu\leq n-1$. $\lambda$ admits matrix representation of type $I_0$ and $I_\nu$ if and only if there exist differentiable functions $L_1:O\longrightarrow O(\nu)$ and $L_2:O\longrightarrow O(n-\nu)$ such that
$$L(a)=\pmatrix{L_1(a)
& 0 \cr 0 & L_2(a)}$$

\end{prop}

\begin{prop}
 Let $\lambda=(N,\psi,O,R,\{e_i\})$ be a s-space over $M$ with $O$ connected. $\lambda$ admits matrix representations of type $I_{\nu}$ for all $0\leq \nu\leq n-1$ if and only if $\lambda$ admits matrix representation of type $A$, for all constant matrix $ A\in \R^{n\times n}$.
\end{prop}
\begin{proof}{{Proof.}}  If  $\lambda$ admits matrix representations of type $I_0,\ I_1,\
\dots, I_{\nu}$, from the  proposition above we have that $L(a)=\pmatrix{\pm 1& & & \cr
& \ddots^{\nu} & & \cr & & \pm 1 & \cr & & & l(a)}$ with $l(a)\in
O(n-\nu)$. Since $L$ is differentiable and $L(ab)=L(a).L(b)$, we see that $L(a)=\pmatrix{Id_{\nu\times \nu} & 0 \cr 0& f(a)}$. If $\nu=n$, then $L\equiv I_{n\times n}$ and the proposition follows.

\end{proof}

\section{Morphisms of s-spaces.}

\begin{defn} Let $\lambda=(N,\psi,O,R,\{e_i\})$ and $\lambda'=(N',\psi',O',R',\{e_i'\})$ be s-spaces over $M$. We call a pair $(f,\tau)$ a morphism of s-spaces between $\lambda$ and $\lambda'$ if
\begin{itemize}

\item [a)] $f:N\longrightarrow  N'$ be differentiable.
\item [b)]$\tau:O\longrightarrow O'$ is a morphism of Lie  groups.
\item [c)]$\psi'\circ f=\psi$.
\item [d)]$f(z.a)=f(z).\tau(a)$ for all $z\in N$ and $a\in
O$.

\end{itemize}

\end{defn}

Note that if $\lambda$ and $\lambda'$ are principal fiber bundles, $(f,\tau)$ is a principal bundle morphism between them.

\begin{exmpl}\label{ejmorfismo1}
Let  $\lambda=(N,\psi,O,R,\{e_i\})$ be a s-space over $M$ and $LM$ the s-space induced by the frame bundle of $M$. Consider the pair $(\Gamma,L):\lambda\longrightarrow LM$, where $\Gamma(z)=(\psi(z),e_1(z),\dots,e_n(z))$ and $L$ is the  base change morphism of $\lambda$, then $(\Gamma,L)$ is a morphism of s-spaces.
\end{exmpl}

\begin{rem} Let $\lambda$ and $\lambda'$ be  s-spaces over $M$ and $(f,\tau):\lambda\longrightarrow\lambda'$ be  a morphism between them. If $\lambda'$ is a principal fiber bundle and $\tau$ is injective, then $\lambda$ is a principal fiber bundle.
\end{rem}
\begin{rem}It is easy too see that if $\tau$ is surjective then $f$ is also surjective. If $O'$ acts without fixed point, then we have that $\tau$ is surjective if and only if $f$ is surjective; the injectivity of $\tau$ implies that of $f$; and if $\tau$ is bijective then so is $f$. If $O$ and $O'$  act without fixed point, then $f$ is injective if and only if $\tau$ is it.

\end{rem}

Let $(f,\tau):\lambda\longrightarrow \lambda'$ be a morphism of s-spaces. As $\psi'(f(z))=\psi(z)$ we have that $\{e'_i(f(z))\}$ and $\{e_i(z)\}$ are bases of $M_{\psi(z)}$. Therefore, there exists $C(z)\in GL(n)$ that  satisfies  $\{e'_i(f(z))\}=\{e_i(z)\}.C(z)$. We called to the function $C:N\longrightarrow GL(n)$ \textit{the linking map of }$(f,\tau)$. For example the linking map of the morphism given in Example \ref{ejmorfismo1} is $C(z)=Id_{n\times n}$. Let $\lambda$ be a s-space over $M$ with base change morphism $L$ and $a_0\in O$. Consider $(f,\tau):\lambda\longrightarrow\lambda$ defined by $f(z)=R_{a_0}$ and $\tau(b)=Ad(a_0^{-1})(b)$, then $C(z)=L(a_0)$.

The linking map of a morphism $(f,\tau)$ satisfies that $C(z.a)=$
$(L(a))^{-1}.C(z).L'(\tau(a))$, where $L$ and $L'$ are the  base change  morphism of $\lambda$ and $\lambda'$ respectively, and the relationship between two linking maps is given by $C_{(g,\gamma)}(z)=C_{(f,\tau)}(z).L'(a(z))$, where $a:N\longrightarrow O$ is a differentiable function.

Let $\lambda=(N,\psi,O,R,\{e_i\})$ be a s-space over $M$ and consider $F:N\longrightarrow \R^{n\times n}$. We say that $F$ comes from a tensor if there exists a tensor $T$ on $M$ such that $^{\lambda}T=F$. In this case, we say that $F$ is the matrix representation (or the induced matrix function by) of $T$ with respect to $\lambda$.

\begin{prop}\label{tensorquevienen}
Let $\lambda=(N,\psi,O,R,\{e_i\})$ and $\lambda'=(N',\psi',O',R',\{e'_i\})$ are s-spaces over $M$ with base change morphism $L$ and $L'$ respectively, and  let $(f,\tau):\lambda\longrightarrow\lambda'$ be a morphism. If $^{\lambda'}T$ is the matrix representation of  $T$ with respect to $\lambda'$, then $^{\lambda'}T \circ f$  comes from a tensor if and only if $$(L(a))^t.(^{\lambda'}T\circ
f)(z).L(a)=(L'(\tau(a)))^t.(^{\lambda'}T\circ f)(z).L'(\tau(a))$$
for all $z\in N$ and $a\in O$.

\end{prop}

\begin{proof}{{Proof.}} If $^{\lambda'}T \circ f$ comes from a tensor, then it satisfies $ (^{\lambda '}T \circ
f)(z.a)=(L(a))^t . ^{\lambda '}(T\circ f)(z).L(a)$. So by definition, we have that $ ^{\lambda '}T (
f(z.a))=L'(\tau (a)))^t. ^{\lambda'}T( f(z)).L'(\tau
(a))$. The other implication  follows by a verification of the invariance property.

\end{proof}

\begin{rem}Let $T$ be a tensor on $M$. From the above Proposition it follows that until the $k^{th}$ iteration of $T$ by $(f,\tau)$ comes from a tensor on $M$ if and only if $L^t.(C^t)^j.^{\lambda}T.C^j.L=$\linebreak
$(L'\circ\tau)^t.(C^t)^j.^{\lambda}T.C^j.(L'\circ
\tau) $ for all $1\leq j\leq k$.
\end{rem}

\begin{cor} The following sentences  are equivalent:
\begin{itemize}
\item[i)]  For all tensor $T$ on $M$, $^{\lambda '}(T\circ f)$ comes from a tensor on $M$.
\item[ii)] $L'\circ \tau=\pm L $.
\end{itemize}

\end{cor}

\begin{prop}\label{mentr} Let $(f,\tau):\lambda \longrightarrow \lambda'$ be a morphism of s-spaces and let  $T$ be a tensor on $M$ then
$$(^{\lambda'}T\circ f)(z)=(C(z))^t.^{\lambda}T(z).C(z)$$
where $C$ is the linking map of $(f,\tau)$.
\end{prop}
\begin{proof}{{Proof.}}
$[(^{\lambda'}T\circ
f)(z)]^i_j=T(\psi'((f(z))))(e_i'(f(z)),e_j'(f(z)))=$

$=T(\psi(z))(\dst\sum_{r=1}^m (C(z))^r_i e_r(z), \dst\sum_{s=1}^m
(C(z))^s_j e_s(z))=\dst\sum_{r,s=1}^m(C(z))^r_i
[^{\lambda}T(z)]^r_s.(C(z))^s_j$

\end{proof}

\begin{defn}
Let $(f,\tau):\lambda\longrightarrow\lambda'$ be a morphism of s-spaces and $T$ be a tensor on $M$. We say that $T$ is invariant by $(f,\tau)$ if $^{\lambda'}T\circ f=^{\lambda}T$. Let us denote with $I_{(f,\tau)}$ the  subspace of $\chi^0_2(M)$ given by the invariant tensors of $(f,\tau)$.
\end{defn}
  For example, let $\lambda$ be  a s-space over $M$, if $(f,\tau):\lambda\longrightarrow LM$ is the morphism given in the Example \ref{ejmorfismo1}, then $I_{(f,\tau)}=\chi^0_2(M)$. Given a s-space $\lambda=(N,\psi,O,R,\{e_i\})$  and  $T\neq 0$, then there exists $a\in GL(n)$ and $z\in N$ such that $a^t.T(z).a\neq T(z)$. Therefore, if we consider the s-space $\lambda'=(N,\psi,O,R,\{e'_i\})$, where $\{e'_i\}=\{e_i\}.a$, we have that  $T$ is not an invariant tensor by the morphism $(Id_N,Id_O)$.

\begin{prop}
Let $(f,\tau):\lambda\longrightarrow \lambda'$ be a morphism and  $T$ be  a tensor on $M$.  If there exists  $k\in \N$ such that the $k^{th}$ iteration by $(f,\tau)$ of $T$ is an invariant tensor, then $T$ is an invariant tensor.

\end{prop}

\begin{proof}{{Proof.}} Let us denoted by $^{\lambda}T^j$ and $^{\lambda'}T^j$ the matrix representation of the $j^{th}$ iteration of $T$ with respect to $\lambda$ and $\lambda'$ respectively. $^{\lambda}T^k=^{\lambda'}T^k\circ f=
C^t.^{\lambda}T^k.C$ , since the $k^{th}$ iteration is an invariant tensor. On the other hand,
$^{\lambda}T^k=(^{\lambda'}T^{k-1}\circ f)=C^t$
 $^{\lambda}T^{k-1}C=C^t.(^{\lambda'}T^{k-2}\circ f) .C=
(C^t)^2.^{\lambda}T^{k-2}.C^2=(C^t)^{k-1}.^{\lambda}T.C^{k-1}$, hence $^{\lambda}T=C^t.^{\lambda}T. C \ .$
\end{proof}

Let $T$ be a tensor on $M$ and $\lambda=(N,\psi,O,R,\{e_i\})$ be a s-space over $M$. For each $z\in N$, consider the lie subgroup of $GL(n)$ defined by  $G_T(z)=\{D\in GL(n): D^t.^{
\lambda}T(z).D=^{\lambda}T(z) \}$.  We call it the \textit{group of invariance of} $T$ at $z$. For simplicity of notation we write $G_T(z)$ instead of $G_T^{\lambda}(z)$ which is more convenient. In these terms, a tensor $T$ is invariant by $(f,\tau)$ if and only if $C(z)\in G_T(z)$  for all $z\in N$.

If $\psi(z)=\psi(z')$ we have $G_T(z)\simeq G_T(z')$, because  $\varphi_a:G_T(z')\longrightarrow G_T(z)$ defined by $\varphi_a(D)=L(a).D.L(a^{-1})=Ad(L(a))(D)$  for  $a\in O$ such that $z'=z.a$, is a homomorphism of Lie groups. We called the subset $F_T=\{(z,g):z\in N \mbox{ and }g\in G_T(z)\}$ of $N\times GL(n)$ \textit{the invariance set of } $T$. If there  is a tensor $T$ on $M$ that admits a matrix representation of the form $^{\lambda}T=\alpha.Id_{n\times n}$, with $\alpha\neq 0$, then $F_T=N\times O(n)$.  Let $\lambda$ be the s-space of Example \ref{ejseLM}. If $T$ is the tensor on $LM$ that satisfies  $^{\lambda}T=\pmatrix{0 &
Id_{m\times m} \cr -Id_{m\times m} & 0}$ here $m=\frac{n+n^2}{2}$, then $F_T=LM\times GL(n)\times{\cal{S}}_m$ where $\cal S$$_m$ denotes the symplectic group of  $\R^{2m\times2m}$. In general $F_T$ does not has a manifold structure. The invariant tensor by a morphism $(f,\tau):\lambda\longrightarrow \lambda'$ they are those that satisfy that $(z,C(z))\in F_T$ for all $z\in N$.
\begin{rem}
Let $(f,\tau):\lambda \longrightarrow \lambda'$ be a morphism with linking map $C$. If $T\in I_{(f,\tau)}$ and $T$ is non degenerated, then $\det(C(z))=\pm1$ for all $z\in N$.
\end{rem}

\section{Connections on s-spaces.}

Given $\lambda=(N,O,\psi,\R,\{e_i\})$ a s-space over $M$, for $z\in N$ let us denote by $V_z$ the vertical subspace at $z$ induced by the projection $\psi$ (i.e. $V_z=\ker \psi_{*_z}$). Note that $\dim V_z=k-s$ where $s$ is the dimension of the stabilizer $S_z$ and $k=\dim O$. As  when we deal with fibrations (see \cite{Mi}), we have a notion of connections for  s-spaces.

\begin{defn} A connection on a s-space  $\lambda$ over $M$  is $(1,1)$ tensor $\phi$ on $N$ that satisfies:
\begin{itemize}
\item[1)]$\phi_z:N_z\longrightarrow V_z$ is a linear map.
\item[2)]$\phi^2=\phi$, $\phi$ is a projection to the vertical subspace.
\item[3)]$\phi_{z.a}((R_a)_{*_z}(b))=(R_a)_{*_z}(\phi(b))$.
\end{itemize}
\end{defn}

Note that $3)$ has sense because $(R_a)_{*_z}(V_z)= V_{z.a}$.

We called to $H_z=\ker \phi_z$ the \textit{horizontal subspace at } $z$. It is clear that $N_z=H_z\oplus V_z$. Since   $\phi_{za}((R_a)_{*_z}(\phi(z)(b)))=(R_a)_{*_z}(\phi(z)(b))=(R_a)_{*_z}(0)=0$, $(R_a)_{*_z}(H_z)=H_{z.a}$. As in the case of connections in principal fiber bundles we have that: There is a connection $\phi$ on $\lambda$  if and only if there exists a differentiable distribution on $N$ ($z\longrightarrow H_z$) such that $N_z=H_z\oplus V_z$ and $H_{z.a}=(R_a)_{*_z}(H_z)$. If we have a distribution with these properties, we define $\phi(z)(b)=b^v$ where $b=b^h+ b^v$.

Let $\lambda=(N,\psi,O,R,\{e_i\})$ be a s-space over $M$ endowed with a connection $\phi$, then we have the concept of {\it{horizontal lift}}.

\begin{defn}  Let $v\in M_p$ and $z\in \psi^{-1}(p)$. We called horizontal lift of $v$ at $z$  to the unique vector $v^h_z\in N_z$ such that $\psi_{*_z}(v^h_z)=v$ and $v^h_z\in H_z$.

\end{defn}

Given a vector field $X$ on $N$ , let $H(X)$ a $V(X)$ the vector fields that satisfy that $H(X)(z)\in H_z$, $V(X)(z)\in V_z$  and $X(z)=H(X)(z)+V(X)(z)$ for all $z\in N$. We called $H(X)$ and $V(X)$ the {\it{horizontal and the vertical projections}} of $X$.  Is easy to see that $H(X)$ and $V(X)$ are smooth vector fields if $X$ is a smooth vector field.


\begin{prop} Let $X$ be a vector field on $M$. Then there exists a unique vector field $X^h$ on $N$  such that  $X^h(z)\in H_z$  and
$\psi_{*_z}(X^h(z))=X(\psi(z))$ for all  $z\in N$.

\end{prop}

\begin{proof}{{Proof.}}
Let $p_0\in M$ and $z_0\in N$ such that  $\psi(z_0)=p_0$. As $\psi$ is a submersion, there exist $(U,x)$ and   $(V,y)$ centered at $p_0$ and $z_0$ respectively  that satisfy $\psi(U)\subseteq
V$ and \linebreak$y\circ\psi\circ
x^{-1}(a_1,\dots,a_n,a_{n+1},\dots,a_m)=(a_1,\dots,a_n)$. If $X(p)=\sum_{i=1}^n\rho^i(p)\frac{\partial}{\partial
y_i}\mid_p$ for $p\in U$, let the vector field on $V$ defined by $\tilde{X}_U(z)=\sum_{i=1}^n(\rho^i\circ\psi)(z)\frac{\partial}{\partial
x_i}\mid_z$, then we have that $\psi_{*}(\tilde{X})=X\circ \psi$.
For this reason, we can take an open covering $ \{U_i \} _{i\in I} $  of $N$ such that
for each  $U_i $ we have a field $ \tilde{X}_i\in
\chi (U_i) $ that satisfies the previous property. Let
$ \{\zeta_i \} _{i\in I} $  be a  unit partition
 subordinate to the covering $\{U_i \} _{i\in
I} $. Consider the vector field $\tilde{X}\in \chi(N) $ given for $\tilde{X}=\sum_{i\in
I}\zeta_i.\tilde{X}_i $. $\tilde{X}$ satisfies that
$\psi_{*_z}(\tilde{X}(z)) = X(\psi(z))$ for all $z\in N $. Finally,  $H(\tilde{X})$ is the vector fields that we  looked for.
The uniqueness follows from the fact that $\psi_{*_z}\mid_{H_z}:H_z\longrightarrow M_{\psi(z)}$ is an isomorphism.

\end{proof}


\begin{rem}

The horizontal distribution $z\longrightarrow H_z$ is trivial since $\{e_i^h(z)=(e_i(z))^h_z\}_{i=1}^n$ is a base of $H_z$ for all $z\in N$ and   $\{e_i^h\}_{i=1}^n$   are  smooth vector fields.

\end{rem}

For all $z\in N$ we have defined the function $\sigma_z:O\longrightarrow N$ given  by $\sigma_z(a)=z.a$. If $X\in {\mathfrak{o}} $, let $V(X)(z)=(\sigma_z)_{*_e}(X)\in V_z$, where $e$ is the unit element of $O$. If the group $O$ acts effectively and $X\neq 0$ is easy to see that $V$ is not the null vector field. If $O$ acts without fixed point, then $V(X)(z)\neq 0$ for all $z\in N$ and $X\neq 0$. Anyway if $\{X_1,\cdots X_{k}\}$ is a base of $\mathfrak{o}$, then $\{V(X_1)(z),\cdots,V(X_k)(z)\}$ spanned  $V_z$. It is not difficult to see that $\ker(\sigma_z)_{*_e}=T_eS_z$.
Consider the 1-forms $\theta_i $ on $N$ defined by $\psi_{*_z}(b)=\sum_{i=1}^n\theta^i(z)(b)e_i(z)$. $\{\theta^1(z),\cdots,\theta^n(z)\}$ are lineally independent and they are a base of the null space of the vertical subspace. Straightforward  calculations show  that the 1-forms $\theta_i$ satisfy that  $L(a).\pmatrix{\theta^1(z.a)((R_a)_{*_z}(b))\cr \vdots \cr
\theta^n(z.a)((R_a )_{*_z}(b))}=\pmatrix{\theta^1(z)(b)\cr \vdots
\cr \theta^n(z)(b) }$ for all $z\in N$ and $a\in O$.

\begin{prop}\label{setrivial}
 Let  $\lambda$  be a s-space over $M$ such that  exists a subspace $\tilde{V}$ of
$\mathfrak{o}$  that satisfies $\dim \tilde{V}=k-s$ ($s=\dim S_z$) and $\tilde{V}\cap T_eS_z=\{0\}$ for all $z\in N$.
If  $\lambda$ admits a connection, then the tangent bundle of $N$ is trivial.
\end{prop}

\begin{proof}{{Proof.}} Let $\{X_1,\dots,X_{k-s}\}$ be a base of $\tilde{V}$, then the vertical  vector fields $V_i(z)=(\sigma_z)_{*_e}(X_i)$ with $i=1,\dots,k-s$ are a base of $V_z$ for all $z\in N$. We have that $\{e_1^h, \dots,e_n^h,V_1,\dots, V_{k-s}\}$ trivialized  the tangent  bundle of $N$.

\end{proof}

\begin{rem} With the same hypothesis of the Proposition,  we a natural dual frame of $N$.
For $i=1,\dots,k-s$, let the 1-forms $W^i$ on $N$ defined by $ \phi_z(b)=\sum_{i=1}^{k-s}W^i(z)(b)V_i(z)$. Then is easy to see that $\{\theta^1(z),\cdots,\theta^n(z), W^1(z),\cdots,W^{k-s}(z)\} $ is a base of $N_z^*$ for all $z\in N$ and it is the dual  base of $\{e_1^h(z),\cdots,e_n^h(z), V_1(z),\cdots,V_{k-s}(z)\}$.
\end{rem}

\begin{rem} Let $\lambda=(N,\psi,O,R, \{e_i\})$ be  a s-space over $M$ that is also a principal fiber bundle. Is well know that every principal fiber  bundle admits a smooth distribution that is transversal to the vertical distribution and is invariant by the action of the group $O$, see  \cite{Du-Fo-No}, so there exists a connection on $\lambda$. On the other hand, the group $O$ acts on $N$ without fixed point and the hypothesis of the Proposition \ref{setrivial} are satisfied. Therefore, the tangent bundle of $N$ is trivial.

\end{rem}

\begin{rem} Let $G$ be a metric on $N$ such that the maps $R_a$ are isometries for all $a\in O$. If $O$ is compact and $N$ is a closed manifold, then  $N$ admits a metric with this property (see \cite{Du-Fo-No}). Let $H_z$ be the subspace of $N_z$ orthogonal to $V_z$. Is easy to see that $z\to H_z$ induces a connection on $\lambda$.

\end{rem}
\begin{rem}In the situation of Proposition \ref{setrivial}, we can lift  a metric $G$ on $M$ to  a metric $\tilde{G}$ on $N$ in a very natural way. Given $G$ a Riemannian metric on $M$ let  $$\tilde{G}=\psi^*(G)+\dst\sum_{i=1}^{k-s}W^{i}\otimes W^{i}.$$
$\tilde{G}$ is a metric on $N$ and $\psi:(N,\tilde{G})\longrightarrow (M,G)$ is a Riemannian submersion.
To keep in mind the metric $\tilde{G}$ can be very useful. For example, using the fundamental equations of a Riemannian submersion \cite{ON} we can relate the curvature tensors of both metrics. Sometimes if we chose appropriately the s-space over $M$, we can simplify considerably the calculation of the curvature tensor of $(M,G)$. This is the case when the base manifolds is the tangent bundle of a Riemnannian manifold. In \cite{H-K}, we use a s-space $\lambda$ and the metric $\tilde{G}$ to compute the curvature tensor of the tangent bundle endowed with certain class of $\lambda$ natural metrics with respect to the bundle.

\end{rem}

\begin{rem}\label{Sasakimok} Let $\lambda$ be a s-space over $M$ and let $\nabla$ be a linear connection on $M$ with connection function $K$. Consider $K^i:TN\longrightarrow TM$ defined by $$K^i_z(b)=K\Big((e_i)_{*_z}(b)\Big)$$
and let  $H_z=\{b\in N_z: K^i_z(b)=0\ \mbox{ for }\ i=1,\dots, n \}$. This smooth distribution is invariant by the group action but it is not necessary complementary to $V_z$. If  $F_z:N_z\longrightarrow M_{\psi(z)}\times \overbrace{M_{\psi(z)}\times \dots\times M_{\psi(z)}}^{n\ times}$ is given by $F_z(b)=(\psi_{*_z}(b),K^1_z(b),\dots,K^n_z(b))$ it is not difficult to see that there are equivalent:
\begin{itemize}
\item[] i) $F_z$ is injective  and $(M_{\psi(z)}\times 0\times \dots\times 0)\in Img\ F_z$.
\item[] ii) $N_z=H_z\oplus V_z$.
\end{itemize}

So if $\lambda$ satisfies $i)-ii)$ we have that $z\to H_z$ induces a connection on $\lambda$.
If $G$ is a metric on $M$ let the (0,2) symmetric tensor on $N$ given by   $$\tilde{G}(A,B) =c(z)G(\psi_{*_z}(A),\psi_{*_z}(B))+\dst\sum_{i=1}^nl_i(z)G(K^i(A),K^i(B))$$ where $c, l_i$ are positive differentiable functions. If $F$ is injective, the $\tilde{G}$ is a Riemmannian metric. If $\lambda$ is the s-space $LM$ and $c=1$ and $l_i=1$ for $i=1,\dots n$, then $\tilde{G}$ is the well know Sasaki-Mok metric (see \cite{Mok} and \cite{CorderoLeon}).
\end{rem}

\section{Natural tensor fields}

\subsection{Natural tensor fields on  fibrations.}
  In this section we will study certain class of tensors  on a manifolds and  fibrations. With a tensor $T$ on a fibration we want to mean that $T$ is a tensor on the space manifold of the fibration.
   If $\alpha=(P,\pi, \F)$ is a fibration  we will consider a particular  class of s-spaces over $P$ in order to  take into account the structure of the fibration for the study of the tensors on it.
\begin{defn} Let $\alpha=(P,\pi, \F)$ be a fibration on $M$ and $\lambda=(N,\psi, O, R, \{e_i\})$ be a s-space over $P$. We say that $\lambda$  is a trivial s-space over $\alpha$ if  $N=N'\times \F$.
\end{defn}

\begin{exmpl} The s-space $\lambda=(LM\times GL(n),\psi, GL(n), R,\{H_i,V^i_j\})$ given in the example \ref{ejseLM} is a trivial s-space over the  frame bundle of $M$.
\end{exmpl}

\begin{defn}
Let $\alpha=(P,\pi,\F)$ be a fibration and $\lambda=(N\times\F,\psi, O, R, \{e_i\})$ be a trivial s-space over $\alpha$.
We say that a tensor $T$ on $P$ is $\lambda$-natural with respect to $\alpha$ if
$^{\lambda}T(z,w)=$$^{\lambda}T(w)$ (i.e. its matrix representation depends only of the parameter $w$ of the fiber $\F$).

\end{defn}

\begin{rem} Let $M$ be a manifold endowed with a linear connection $\nabla$ and a Riemannian metric $g$. If we consider the s-spaces $\lambda=(LM\times GL(n),\psi, GL(n), R,\{H_i,V^i_j\})$ (Example \ref{ejseLM}) and  $\lambda'=(O(M)\times GL(n),\psi, O(n), R,\{H_i,V^i_j\})$, where $O(M)$ is the manifold of orthonormal bases of $(M,g)$ and the action of  the orthonormal group and the projection are similar to that ones of $\lambda$, then the concept of $\lambda-natural$ and $\lambda ' -natural$ with respect to $(LM,\pi, GL(n))$ agree  with that ones  of natural tensor with respect to the connection $\nabla$ and with respect to the metric $g$   given in \cite{GK}.

\end{rem}

\begin{rem}
 There exist s-spaces  such that the concept of $\lambda-natural$ with respect to the fibration agree with the known cases of naturality. So, our definition also generalizes the notion of natural tensor on the tangent and the cotangent bundle of a Riemannian  (see \cite{MCC-GK} and  Example \ref{ejsubsuper1}) and semi-Riemannian
manifold (see \cite{A-GK}).
\end{rem}

\subsection{Natural tensor fields on manifolds.}

In view of the definition of $\lambda-natural$ with respect to a fibration, it seems interesting  to ask what it means to be $\lambda-natural$ with respect to a manifold. A manifold $M$ can be view as a trivial fibration $\alpha_M=(M\times\{a\},pr_1, \{a\})$ and there is a one to one correspondence between the s-spaces over $\lambda$ and the trivial s-spaces over $\alpha$. A s-space $\lambda=(N, \psi, O, R, \{e_i\})$ over $M$ induced the
 $\lambda'=(N\times\{a\},\psi, O, R, \{e_i\})$ over $\alpha$. A tensor $T$ on $M$ induce a tensor $T'$ on $M\times\{a\}$. Then $T'$ is $\lambda'-natural$ with respect to a $\alpha$ if and only if $^{\lambda'}T'(z,a)=^{\lambda'}T'(a)$, therefore $T'$ is  $\lambda'-natural$ with respect to a $\alpha$ if and only if $^{\lambda}T$ is a constant map. This suggests  the following definition:

\begin{defn}
Let $\lambda$ be a s-space over $M$ and $T$ a tensor on $M$. We say that $T$ is $\lambda-natural$ if $^{\lambda}T$ is a constant map.
\end{defn}

\begin{exmpl}
Let $(M,g)$ be a Riemannian manifold and let $\lambda=(O(M),\pi,O(n),\cdot,\{\pi_i\})$ the s-space over $M$ induced by the orthonormal frame bundles of $M$. Since $L(a)=a$ for all $a\in O(n)$, $T$ is $\lambda-natural$ if and only if $^{\lambda}T=k.Id_{n\times n}$, that is $T$ is an scalar multiple of the metric $g$.
\end{exmpl}

\begin{exmpl} Suppose that the map $F$ of the Remark \ref{Sasakimok} is bijective. Let $\beta=(N,id_N,\{1\},$ $,(\ \cdot\ ), \{(e_i(z))^h, (e_j(z))^{v(i)}_z\})$ be the s-space over the space manifold of $\lambda$, where $\{1\}$ is the trivial group and $(\ \cdot\ )$ is the trivial action, $(e_i(z))^h$ is the horizontal lift of $e_i(z)$ at $z$ and   $(e_j(z))^{v(i)}_z$ satisfies that $K^i((e_j(z))^{v(i)}_z)=e_j(z)$. If $G$ is a metric on $M$ and $\tilde{G}$ is the generalizes Sasaki-Mok metric on $N$ then $$^{\beta}\tilde{G}(z)=\pmatrix{[^{\lambda}G] & 0 &\cdots &0\cr 0 & [^{\lambda}G]& 0&  0\cr
0 & 0 & \ddots & 0 \cr 0& \cdots & \cdots & [^{\lambda}G]}$$
so $\tilde{G}$ is $\beta$ natural if and only if $G$ is $\lambda$-natural.
\end{exmpl}

\begin{rem} Let $\alpha=(P,\pi,\F)$ be a fibration on $M$ and $\lambda$ a trivial s-space over $\alpha$. $\lambda$ is also a s-space over $P$. If a tensor $T$ on $P$ is $\lambda-natural$ then $T$ is $\lambda-natural$ with respect to $\alpha$. The converse implication not necessarily holds. Let $\lambda=(O(M)\times GL(n),\psi,O(n),R,\{H_i,V^i_j\})$ over $LM$,  there are more $\lambda-natural$ tensors with respect to $LM$ than constant maps, see \cite{GK}.

\end{rem}

\begin{rem} Consider the s-space $LM$  and let  $T$ be a $LM-natural$ tensor on $M$. Let $A\in R^{n\times n}$ such that $^{LM}T\equiv A$. Since the base change morphism of $LM$ is the identity of $GL(n)$,   $A=a^t.A.a$ for all $a\in GL(n)$, hence  $T$ must be the null tensor. Therefore, for a manifold $M$ the null tensor is the only one  that is $\lambda-natural$ for all the s-spaces over $M$.
\end{rem}

\begin{rem}
If $T$ is $\lambda-natural$, we have that  $N\times Im(L)\subseteq F_T$ where $F_T=N\times G$ with $G$ a subgroup of $GL(n)$.
\end{rem}

Let $\lambda=(N,O,\psi,\R,\{e_i\})$ be a s-space over $M$ . Note that if $T$ is $\lambda-natural$ and $(f,\tau):\lambda\longrightarrow\lambda$ is a morphism of s-spaces then $T\in I_{(f,\tau)}$. In the other hand, if $T\in I_{(f,\tau)}$ for all $(f,\tau)$ automorphism of $\lambda$, then $^{\lambda}T$ is constant in each fiber of $N$.  A necessary and sufficient condition for a tensor $T$ to have a constant matrix representation  in each fiber is that $T\in I_{(f_a,\tau_a)}$ for all $a\in O$, where $(f_a,\tau_a)$ is the morphism defined by $f_a(z)=R_a(z)$ and $\tau_a(b)=a^{-1}b.a$. Let us see some facts  about the  relationship between the natural tensors and the morphisms of s-spaces.  The next two  proposition follow from Proposition \ref{mentr}.

\begin{prop}
Let $\lambda$ and $\lambda'$ be two s-spaces over $M$ and $(f,\tau):\lambda\longrightarrow\lambda'$ be a morphism with linking map $C$. If $T$ is a  $\lambda'-natural$ tensor with  $^{\lambda'}T= A\in \R^{n\times n}$, then $T$ is $\lambda-natural$ if and only if $(C(z)^{-1})^t.A.C(z)^{-1}$ is a constant map.
\end{prop}

\begin{prop} Let $(f,\tau):\lambda\longrightarrow \lambda'$ be a morphism of s-spaces with linking map $C$ and $T$ a tensor on $M$  that is $\lambda$ and $\lambda'- natural$. Let $A$ and $B\in R^{n\times n}$ such that $^{\lambda}T=A$  and $^{\lambda'}T=B$, then $C(z)^t.A.C(z)=B$ for all $z\in N$.

\end{prop}
 In particular, if $\lambda=\lambda'$  the image of the linking map of any automorphism have to be included in the  group of invariance of  all the $\lambda-natural$ tensors. For example, if $\lambda=(LM\times GL(n),\psi, GL(n), R,\{H_i,V^i_j\})$ and $(f,\tau)$ is an automorphism of $\lambda$ with linking map $C$, then $C(z)=Id_{(n+n^2)\times (n+n^2)}$ for all $z\in LM\times GL(n)$.

\begin{prop}
Let $\lambda=(N,\psi,O,R,\{e_i\})$ and $\lambda'=(N',\psi ',O',R',\{e_i'\})$ be two s-spaces over $M$, $(f,\tau):\lambda\longrightarrow\lambda'$  be a morphism of s-space, $T$ a $\lambda'-natural$ tensor and let $A\in R^{n\times n}$  such that $^{\lambda'}T=A$. Then $^{\lambda'}T\circ f$ comes from a tensor on $M$ if and only if $(L(a))^t.A.L(a)=A$ for all $a\in O$.
\end{prop}

\begin{proof}{{Proof.}}  Since $T$ is $\lambda'-natural$,  $(L'(a'))^t.A.L'(a')=A$ for all $a'\in O'$, then the Proposition  follows from  Proposition \ref{tensorquevienen}.

\end{proof}

\begin{rem}
 There are tensors on $M$ that are not $\lambda-natural$ for any $\lambda$ s-space over $M$. Let $T$ be a not null tensor on $M$, then there exists $p\in M$ such that $T(p):M_p\times M_p\longrightarrow \R$ is not the null bilinear form. Let $f$ be a differentiable function on $M$ that satisfies $f(p)=1$ and $f(q)=0$ for a point  $q$ different of $p$ and consider the tensor $\widetilde{T}$ defined by $\widetilde{T}(\xi)=f(\xi).T(\xi)$. If $\tilde{T}$ is $\lambda-natural$, then $^{\lambda}\tilde{T}\equiv A$ and since $\widetilde{T}(q)=0$, $A$ must be the  zero matrix. But for $z'\in \psi^{-1}(p)$, we have that $^{\lambda}\widetilde{T}(z')=[\widetilde{T}(q)(e_i(z'),e_j(z'))]=f(p)[T(p)(e_i(z').e_j(z'))]\neq 0$, hence $T$ is not $\lambda-natural$.

\end{rem}

\begin{prop}Let $T$ be a symmetric tensor on $M$ with index and rank constant, then there is a s-space $\lambda$ over $M$ such that  $T$ is $\lambda-natural$.

\end{prop}

\begin{proof}{{Proof.}}
If $rank(T)=0$ then $T$ is the null tensor and $T$ is $\lambda-natural$ for all $\lambda$. Suppose that $rank(T)=r\geq 1$ and $index(T)=r-s$.  For every  $p\in M$ there is a base $\{v_1,\dots,v_s,v_{s+1},\dots,v_r,v_{r+1},\dots,v_n\}$ of
$M_p$ that diagonalizes the matrix of $T(p)$, that is
 $[T(p)(v_i,v_j)]=$
 $\pmatrix{Id_{s\times s} &
0 & 0 \cr 0 & -Id_{(r-s)\times (r-s)} & 0 \cr 0 & 0 & 0 }=I_{sr}$. Let  $\lambda=(N,\pi,O,\cdot,\{\pi_i\})$ where $N=\{(q,v)\in LM: [T(q)(v_i,v_j)]=I_{sr}\}$,
$O=\pmatrix{O(s)& 0 & 0 \cr 0 & O(r-s) & 0  \cr 0 & 0 &
GL(n-r)}$ and the action, the projection and the map $\{\pi_i\}$ are similar to those of $LM$. Then $^{\lambda}T=I_{sr}$
\end{proof}

\section{Subs-spaces.}

Let $\lambda=(N,\psi,O,R,\{e_i\})$ and $\lambda'=(N',\psi', O',R',\{e_i'\})$ be s-spaces over $M$ and $N$ respectively and $h:M\longrightarrow M'$ be a differentiable function. Let $f:N\longrightarrow N'$ be a differentiable function and $\tau:O\longrightarrow O'$  a group morphism.

\begin{defn} We said that $(f,\tau)$ is $a$ morphism of s-spaces over
 $h$ if $f(z.a)=f(z).\tau(a)$ for all $z\in N$ and $a\in O$ and
 $\psi'\circ f=h \circ\psi$.
\end{defn}

This definition generalize the concept of morphism of s-spaces. If $\lambda$ and $\lambda'$ are s-spaces over $M$, and $(f,\tau):\lambda\longrightarrow \lambda'$ is a morphism of s-spaces then $(f,\tau)$  is a morphism over  $Id_M$.

\begin{exmpl}\label{ejsubsuper1} Let $(M,g)$ be a Riemannian manifold and let $\lambda=(O(M)\times R^n,\psi, O(n),R,\{e_i\})$ the s-space over $TM$  where
the projection is defined by $\psi(p,u,\xi)=(p,\sum_{i=1}^nu_i\xi^i)$ and the action of the orthonormal group on $O(M)\times \R^n$  is given by  $R_a(p,u)=(p,u.a,\xi.a)$. For $1\leq i \leq n$, let  $e_i(p,u,\xi)=(\pi_{*_{\psi(p,u,\xi)}}\times
K_{\psi(p,u,\xi)})^{-1}(u_i,0)$  and
$e_{n+i}(p,u,\xi)=(\pi_{*_{\psi(p,u,\xi)}}\times
K_{\psi(p,u,\xi)})^{-1}(0,u_i)$,  where $K$ is the connection map induced by the Levi-Civita connection of $g$. Before we see an example of subs-space let us make a brief comment. The tensor on $TM$ that are $\lambda$ natural with respect to $TM$ agree with the ones of Calvo-Keilhauer \cite{MCC-GK}. The Sasaki $G_S$ and the Cheeger-Gomoll $G_{cg}$  metric are $\lambda-natural$ with respect to $TM$. The matrix representation of the Sasaki metric and the Chegeer-Gromoll metric are $^{\lambda} G_S(p,u,\xi)=\pmatrix{Id_{n\times n} & 0 \cr 0 & Id_{n\times n}}$ and $^\lambda G_{cg}(p,u,\xi)=\pmatrix{Id_{n\times n} & 0 \cr 0 &
\frac{1}{1+|\xi|^2}(Id_{n\times n}+(\xi)^t.\xi)}$ respectively .

Consider the s-space $\lambda'=(O(M), \psi',O(n-1),R', \{e_i'\})$ over the unitary tangent $T_1M$ bundle of $M$, where $\psi'(p,u)=(p,u_n)$, The action of $O(n-1)$ on $O(M)$ is given by  $R'_a(p,u)=(p,\sum_{i=1}^{n-1}u_ia^{i}_1,\dots,\sum_{i=1}^{n-1}u_ia^{i}_{n-1},u_n)$. The maps $\{e_i'\}$ are defined by  $e'_i(p,u)=(\pi_{*_{\psi(p,u)}}\times
K_{\psi(p,u)})^{-1}(u_i,0)$ if $1\leq i \leq n$ and
by $e'_{n+i}(p,u)=(\pi_{*_{\psi(p,u)}}\times
K_{\psi(p,u)})^{-1}(0,u_i)$ if $1\leq i\leq n-1$. Let $f:O(M)\longrightarrow O(M)\times \R^n$ and $\tau:O(n-1)\longrightarrow O(n)$ defined by
$f(p,u)=(p,u,v)$ where $v$ is the $n^{th}$ vector of the canonic base of $R^n$,  and $\tau(a)=\pmatrix{a & 0 \cr 0 & 1}$. $(f,\tau):\lambda\longrightarrow\lambda'$ is a morphism of s-spaces over the inclusion map of $T_1M$ in $TM$.
\end{exmpl}

Let $M$ and $M'$ be manifolds of dimension $n$ and $n'$ respectively. Let  $\lambda=(N,\psi,O,R,\{e_i\})$ and $\lambda'=(N',\psi',O',R',\{e_i'\})$ be s-spaces over $M$ and $M'$ and $(f,\tau):\lambda\longrightarrow
\lambda'$ a morphism of s-space over an  inmersion $h:M\longrightarrow M'$. For every $z\in N$,  $h_{*_{\psi(z)}}(M_{\psi(z)})$ is a subspace of dimension $n$ of $M'_{\psi(f(z))}$ and it is generated by $\{h_{*_{\psi(z)}}(e_1(z)),\dots,h_{*_{\psi(z)}}(e_{n}(z))\}$. As $\{e_i'(f(z))\}$ is a base of  $M'_{\psi'(f(z))}$, for every $z\in N$ there exists a matrix $A(z)\in R^{n'\times n'}$ with $rank(A(z))=n$ that satisfies
$$\{h_{*_{\psi(z)}}(e_1(z)),\dots,h_{*_{\psi(z)}}(e_{n}(z)),\overbrace{0,\dots,0}^{n'-n}\}=\{e_1'(f(z)),\dots,e_{n'}'(f(z))\}.A(z)$$
In the previous example,  $A(p,u)=\pmatrix{Id_{(2n-1)\times (2n-1)} & 0 \cr 0 & 0}$.
 If $M=M'$ and $h$ is the identity map then $(f,\tau)$ is a morphism of s-spaces and $A(z)=C^{-1}(z)$ is $C$ is the linking map of $(f,\tau)$.

In this situation, we have the following definition:

\begin{defn}
$\lambda$ is a subs-space of $\lambda'$ if there exists a morphism of s-spaces $(f,\tau)$ over an injective  inmersion $h:M\longrightarrow M'$ such that $f$ is an inmersion and the map $A$ induced by $(f,\tau)$ is constant. In this case, we said that $\lambda$ is a subs-space of $\lambda'$ with morphism $(f,\tau)$ over  $h$.  A s-space $\lambda=(N,\psi,O,R,\{e_i\})$  is  included  in $\lambda'=(N',\psi',O',R',\{e_i'\})$ if $N\subseteq N'$.

\end{defn}

\begin{exmpl}
 Let $M$ be a parallelizable manifold , $V$ a vectorial space and $V'$  a subspace of $V$. Let $GL(V)$ the group of linear isomorphisms of $V$ and $GL(V,V')$ the subgroup of linear isomorphisms of $V$ with the property that $T(V')=V'$. Consider the s-space $\lambda=(M\times V, pr_1, GL(V), R_{f}, \{e_i\})$ over $M$, where the action is defined by  $R_f(p,z)=(p,f(z))$ for $(p,z)\in M\times V$  and $f\in GL(V)$, and $e_i=\bar{e}_i\circ pr_1$ where
$\{\bar{e}_1,\cdots,\bar{e}_n\}$ are the vector fields that trivialized the tangent bundle of $M$. If $\lambda'=(M\times V',pr_1,GL(V,V'), R_f ,\{e_i\} )$, then $\lambda'$ is a subs-space of $\lambda$.

\end{exmpl}

\begin{prop} Let $\lambda=(N,\psi,O,R,\{e_i\})$ and $\lambda'=(N',\psi',O',R',\{e_i'\})$ be s-spaces over $M$ such that $\lambda$ is a subs-space of $\lambda'$ with  morphism $(f,\tau)$  over the identity map of $M$. If  a tensor $T$ on $M$ is $\lambda'-natural$ then $T$ is $\lambda-natural$.
\end{prop}
\begin{proof}{{Proof.}}$[^{\lambda
}T(z)]_{ij}=T(\psi(z))(e_i(z),e_j(z))=T(\psi'(f(z)))(\sum_{l=1}^ne'_l(z)A^l_i,\sum_{s=1}^ne'_s(z)A^s_j)=$

$=\sum_{ls}^n A_i^l.A^s_j [^{\lambda'}T]_{ij}$, then $^{\lambda
}T$ is a  constant map.

\end{proof}

\begin{rem} The converse statement does not holds in general. Let $(M,g)$ be a Riemannian manifold and  $O(M)$ be the s-space  induced by the principal bundle of orthonormal frames. If $i_{O(M)}:O(M)\longrightarrow LM$ and $i_{O(n)}:O(n)\longrightarrow GL(n)$ are the respective inclusion, then $O(M)$ is a subs-space of $LM$ with morphism $(i_{O(M)},i_{O(n)})$ over the identity map of $M$. We known that there are $O(M)-natural$ tensors that are not $LM-natural$.

\end{rem}

Let $T$ be a tensor on $M$ and let  $^{LM}T:LM\longrightarrow
R^{n\times n}$ the matrix map induced by the s-space $LM$. Given  a s-space $\lambda=(N,\psi,O,R,\{e_i\})$ over $M$ we have a morphism $(\Gamma,L):\lambda\longrightarrow LM$, see Example \ref{ejmorfismo1}. It is clear that   $^{\lambda}T=^{LM}T\circ \Gamma$, thus if $T$ is $\lambda-natural$ then there exists a matrix $A\in \R^{n\times n}$ such that  $Img \Gamma
\subseteq (^{LM}T)^{-1}(A)$.

\begin{prop} Let $T$ be a tensor on $M$.
There exists $\lambda$ a s-space over $M$ such that  $T$ is $\lambda-natural$ if and only if there exist a matrix $A\in R^{n\times n}$ and a subs-space of $LM$ included in $(^{LM}T)^{-1}(A)$.

\end{prop}

\begin{proof}{{Proof.}}
Suppose that $T$ is $\lambda- natural$ ( $\lambda=(N,\psi,O,R,\{e_i\})$) and let $A\in R^{n\times n}$ such that $^{\lambda}T=A$. Let $\lambda'=(\Gamma(N),\pi,L(O),R', \{\pi_i\} )$, where $\pi$, $R'$ and $\{\pi_i\}$ are induced by $LM$.    The map              $\pi:\Gamma(N)\longrightarrow M$ is a submersion. Since $\pi(\Gamma(N))=\psi(N)=M$, $\pi$ is surjective. Let $p\in M$ and $z\in \psi^{-1}(p)$, then $\pi(\Gamma(z))=p$. We  are going to see that $\pi_{*_{\Gamma(z)}}:N_{\Gamma(z)}\longrightarrow M_p$ is surjective. Given $v\in M_p$ there exists $w\in N_z$ such that $\psi_{*_z}(w)=v $. Let $\alpha$ be a curve on $N$ that satisfies $\alpha(0)=z$ and $\dot{\alpha}(0)=w$, then for $\beta(t)=\Gamma(\alpha(t))$ we have that
$\beta(0)=\Gamma(z)$ and $\pi_{*_{\Gamma(z)}}(\dot{\beta}(0))=D|_0(\pi(\beta(t)))=\psi_{*_z}(w)=v$. In the other hand, it is clear that $L(O)$ acts transitively on $\Gamma (N)$, so $\lambda'$ is a s-space and it is a subs-space of $LM$ with morphism $(i_{\Gamma(N)},i_{L(O)})$ over the identity map of $M$.

Conversely, suppose that there exist $A\in R^{n\times n}$ and $\lambda=(N,\psi,O,R\{e_i\})$ a s-space over $M$ that is also a subs-space of $LM$ with morphism $(f,\tau)$ over the identity map,  and it holds that $f(N)\subseteq (^{LM}T)^{-1}(A)$. Since $\{e_i(z)\}=\{\pi_i(f(z))\}.B$ for $B\in GL(n)$,   $[^{\lambda}T(z)]=[T(\psi(z))(e_i(z),e_j(z))]=B^t.[T(\psi(z))(\pi_i(f(z)),\pi_j(f(z)))].B=B^t.A.B$

\end{proof}

\section{Atlas of s-spaces.}

\begin{defn}
Let $M$ be a manifold and let  ${\cal{A}}:\{
\lambda_i=(N_i,\psi_i,O_i,R_i,\{e_l\}) \}_{i \in I}$ be a collection of s-spaces over $M$. The collection  ${\cal{A}}$ is called an Atlas of s-spaces if for each pair  $(i,j)\in I\times I$ there is a morphism of s-spaces $(f_{ij},\tau_{ij}):\lambda_i\longrightarrow \lambda_j$ such that $f_{ij}:N_i\longrightarrow N_j$
 is a diffeomorphism.

\end{defn}

We said that the s-spaces $\lambda$ and $\beta$ are {\textit{compatible}} if there exists a morphism $(f_{\lambda \beta},\tau_{\lambda,
 \beta}):\lambda\longrightarrow \beta $ and  $(f_{\beta \lambda},\tau_{\beta,
 \lambda}):\beta\longrightarrow \lambda $ such that  $f_{\lambda
 \beta}$ and $f_{\beta \lambda}$ are diffeomorphisms. Hence, an atlas is a set of compatible s-spaces over $M$.  If ${\mathcal{A}}$ satisfies that for an atlas ${\cal{B}}$, ${\cal{A}}\subseteq{\cal{B}}$  implies ${\cal{A}}={\cal{B}}$, we  called it a  {\textit{maximal atlas}}. In other words, if $\lambda$ is a s-space compatible with the s-spaces of ${\cal{A}}$ then $\lambda\in {\cal{A}}$. If $\lambda$ is a s-space over $M$ let us notate with $\cal{A}=<\lambda>$ the maximal atlas generated by $\lambda$. Let ${\cal{A}}$ be a maximal atlas, it follows from the definition that  ${\cal{A}}=<\lambda>$ for every $\lambda\in {\cal{A}}$. Note that there are different maximal atlases over a  manifold. Consider a metric on $M$, then $<LM>$ and $<O(M)>$ are maximal s-spaces but they  are different because $LM$ and $O(M)$ are not compatible.

Let $\lambda$ be a s-space over $M$, then ${\cal{A}}=\{\lambda\}$ is an atlas. Therefore the concept of atlas is a generalization of the notion of s-space.

\begin{exmpl}\label{ejatlas1}Let $\lambda=(N,\psi,O,R,\{e_i\})$ be a s-space over $M$ and let $A:N\longrightarrow GL(n)$ be a differentiable function. Consider $\lambda_A=(N,\psi,O,R,\{e_l^A\})$ where $e^A_l(z)=\sum_{i=1}^ne_i(z)A^i_l(z)$. The collection ${\cal{A}}=\{\lambda_A\}_{A\in {\cal{F}}(M)}$ is an atlas of s-spaces.

\end{exmpl}

\begin{exmpl} Let $M$ be a parallelizable manifold and $\{H_i\}_{i=1}^n$ the vector fields that trivialized the tangent bundle of $M$. Let $(N,g)$ be a Riemannian manifold such that its isometry group $I_{(N,g)}$ acts transitively on $N$. Let $\lambda_{(N,g)}=(M\times N, pr_1, I_{(N,g)}, R_f, \{ H_i\circ
pr_1\})$ where the action of $I_{(N,g)}$ on $M\times N$ is given by  $R_f (z,p)=(z,f(p))$. If $(N',g')$ is isometric to $(N,g)$ then $\lambda_{(N,g)}$ is compatible with $\lambda_{(N',g')}$. If $N'$ is not diffeomorphic to $N$ then  $< \lambda_{(N,g)}>$ and $ <\lambda_{(N',g')}>$ are different maximal atlas of s-spaces over $M$.

\end{exmpl}

\begin{defn}
Let ${\cal{A}}$ and  ${\cal{B}}$ be two atlases of s-spaces over $M$ and $F$ a collection of morphisms of s-spaces from a s-space of ${\cal{A}}$
to a s-space of ${\cal{B}}$. $F$ will be called a morphism between the atlas ${\cal{A}}$ and ${\cal{B}}$ if for every $\lambda\in {\cal{A}}$  and $\beta\in {\cal{B}}$ there exist $(f,\tau)\in F$ such that $(f,\tau):\lambda\longrightarrow\beta $.
\end{defn}

\begin{rem}
Let ${\cal{A}}$ and ${\cal{B}}$ be two atlas  over $M$, $\lambda_0\in {\cal{A}}$, $\beta_0\in {\cal{B}}$ and $(f_0,\tau_0):\lambda_0\longrightarrow\beta_0$. Consider $\dst F=\{f_{\beta_0
\beta}\circ f_0\circ f_{\lambda
\lambda_0},\tau_{\beta_0\beta}\circ\tau_0\circ \tau_{\lambda
\lambda_0, }\}_{\lambda\in {\cal{A}},\
 \beta\in {\cal{B}}}$ where $(f_{\beta_0
\beta},\tau_{\beta_0\beta}):\beta_0\longrightarrow \beta$ and
$(f_{\lambda
\lambda_0},\tau_{\lambda\lambda_0}):\lambda\longrightarrow
\lambda_0$ are the morphisms that show the compatibility between  $\beta$ and $\beta_0$  and  between $\lambda$ and $\lambda_0$.  $F$ is morphism of atlases between ${\cal{A}}$ and  ${\cal{B}}$.

\end{rem}

\begin{rem}
If $\lambda$ is a s-space over $M$ we have  a canonical  morphism $(\Gamma_{\lambda},L_{\lambda}):\lambda\longrightarrow LM$ (see Example \ref{ejmorfismo1}), hence for every s-space $\lambda$  we have a morphism between the atlases  $<\lambda>$ and  $<LM>$. It seems interesting  to ask if this property characterized $<LM>$. In other words, if a s-space $\beta $ satisfies that for every $\lambda$ there exists a morphism $(f_{\lambda},\tau_{\lambda}):\lambda\longrightarrow \beta$  it has to be necessarily compatible with $LM$?
$$
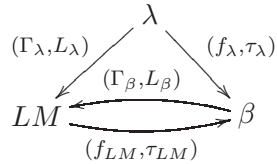
\xymatrix{&
\lambda\ar[dl]_{(\Gamma_{\lambda},L_{\lambda})}\ar[dr]^{(f_{\lambda},\tau_{\lambda})} & \\
LM\ar@/_/[rr]_{(f_{LM},\tau_{LM})}& &
\beta\ar@/_/[ll]_{(\Gamma_{\beta},L_{\beta})}}$$

The answer is  no. Consider a  parallelizable Riemannian manifold $(M,g)$. Let $\{H_i\}_{i=1}^n$ be orthonormal fields that trivialized the tangent bundle of $M$. If $\lambda=(N,\psi,O,R,\{e_i\})$ is a s-space over $M$ let $(f_{\lambda},\tau_{\lambda}):\lambda\longrightarrow O(M)$ defined by $f(z)=(\psi(z),H_1(\psi (z)),\dots,H_n(\psi (z)))$ and
$\tau(a)=Id_{n\times n}$. Therefore, for every maximal atlas ${\cal{A}}$ there is a morphism between it and $O(M)$, and we just know that $O(M)$ it is not compatible with $LM$.
$$
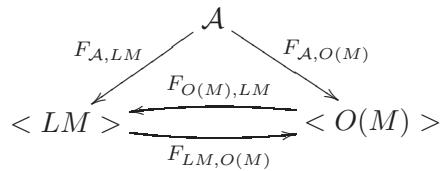
\xymatrix{ & {\cal{A}}\ar[dl]_{F_{{\cal{A}},LM}}\ar[dr]^{F_{{\cal{A}},O(M)}} & \\
<LM>\ar@/_/[rr]_{F_{LM, O(M)}}& &
<O(M)>\ar@/_/[ll]_{F_{O(M),LM}}}$$

But there are more atlases with this property. If $(M,g)$ is an oriented manifold, the maxi-mal atlas generated by the s-space induced by the principal fiber bundles of orthonormal oriented bases $SL(M)$ have this property.  The atlas   $<(M,Id_M, \{1\},
R_1,\{ H_i \})>$, where $R_1$ is the trivial action, is another example.
\end{rem}

\begin{defn} Let ${\cal{A}}$ be an atlas of s-spaces over $M$. A  tensor $T$  on $M$ will be called ${\cal{A}}-natural$ if $T$ is $\lambda-natural$ for all $\lambda\in{\cal{A}}$.

\end{defn}

Note that the concept of ${\cal{A}}- naturality$ generalized the notion of $\lambda-naturality$. If we considerer the atlas ${\cal{A}}=\{\lambda\}$ then $T$ is ${\cal{A}}-natural$ if and only if $T$  is $\lambda-natural$.

\begin{exmpl}\label{ejatlas2}
Let $\lambda$ be a s-space over $M$ and consider the subatlas of the  atlas given in the Example \ref{ejatlas1} defined by ${\cal{A}}=\{\lambda_A\}_{A\in GL(n)}$. Then $T$ is ${\cal{A}}-natural$ if and only if $T$ is $\lambda-natural$. Let $T$ be a $\lambda-natural$ tensor on $M$ and ${\cal{A'}}=\{\lambda_A\}_{A\in {\cal{F}}(N,G_{T})}$, then $T$ is ${\cal{A'}}- natural$ and it has the same matrix representation in all the s-spaces of the atlas.
\end{exmpl}

\begin{rem} If ${\cal{A}}$ is a maximal atlas then the unique ${\cal{A}}-natural$ tensor is the null tensor. Let $\lambda=(N,\psi,O,R,\{e_i\})\in {\cal{A}}$ and $f:N\longrightarrow\R$ be a differentiable function such that $f(z)\neq 0$ for all $z\in N$ and $f^2$ is not constant.  If $\lambda'=(N,\psi,O,R,\{f.e_i\})$, hence $\lambda'\in {\cal{A}}$,  but the null tensor is the only $\lambda-natural$ and $\lambda'-natural$ at the same time, therefore $T\equiv 0$ is the unique ${\cal{A}}-natural$ tensor.
\end{rem}

\begin{defn}

Let ${\cal{A}}$ be an atlas of s-spaces over $M$ and $T$ a tensor on $M$. $T$ is called ${\cal{A}}-weak\ natural$  if there exists $\lambda\in {\cal{A}}$ such that $T$ is $\lambda-natural$.

\end{defn}

If $\cal{A}=\{\lambda\}$ or $\cal{A}$ is the atlas of Example \ref{ejatlas2}, then the concept of ${\cal{A}}-natural$ and ${\cal{A}}-weak\ natural$  coincide.

For study the naturality of tensors on a fibration $\alpha=(P,\pi, \F)$ it will be useful consider the atlases $\cal{A}$ such that all its s-spaces are trivial over $\alpha$. An atlas with this property will be called a\textit{ trivial atlas } over $\alpha$. The following definition is a generalization of the concept of naturality with respect to a fibration:
\begin{defn}
Let $\cal{A}$ be a trivial atlas over a fibration $\alpha=(P,\pi,\F)$ and $T$ a tensor on $P$, then $T$ is ${\cal{A}}-natural$ with respect to $\alpha$ if $T$ is $\lambda-natural$ with respect to $\alpha$  for all $\lambda\in \cal{A}$.

\end{defn}

\begin{exmpl}
Let $\alpha=(P,\pi,G,\ \cdot\ )$ be a principal fiber bundle on $(M,g)$ endowed with a connection $\omega$. For every $W=\{W_1\cdots,W_k\}$ base of $\mathfrak{g}$ let  $\lambda_{W}=(N,\psi,O,R,\{e_i^W\})$ where $N=\{(p,u,b):p \in P,\ u\ \mbox{is an orhonormal base of}\  M_{\pi(p)}, \ b \in G\}$, $\psi(q,u,b)=q.b$, \linebreak$O=O(n)\times G$ and the action $R$ is defined by $R_{(h,a)}(q,u,b)=(qa,uh,a^{-1}b)$.  For $1\leq i \leq n$,  $e_i^W(p,u,g)$ is the horizontal lift of  $u_i$ with respect to $\omega$  at $p.g$  and for $1\leq j\leq k$,  $e_{n+j}(p,u,g)$ is the only one vertical vector on $P_{p.g}$ such that $\omega(p)(e_{n+j}(p,u,g))=W_j$. ${\cal{A}}=\{\lambda_W\}_{W\in L\mathfrak{g}}$ is a trivial atlas over $\alpha$. An  easy computation shows that the set of ${\cal{A}}-natural$ tensor with respect to $\alpha$ are all of those that  there exists $\lambda_W$
 such that $T$ has a matrix representation of the form $^{\lambda_W}T(q,u,a)=\pmatrix{f(a).Id_{n\times n} & 0 \cr 0 &
B(a)}$, where $f:G\longrightarrow \R$ and  $B:G\longrightarrow
\R^{k\times k}$ are differentiable functions.

\end{exmpl}

As above, if $\cal{A}$ is a maximal trivial atlas over $\alpha$ the only ${\cal{A}}-natural$ tensor  with respect to $\alpha$ is the null tensor. So we have a weak definition of naturality for this case too. We said that $T$ is ${\cal{A}}-weak\ natural$ with respect to $\alpha$ if $T$ is $\lambda-natural$ with respect to $\alpha$ for some $\lambda\in {\cal{A}}$.

\section{Examples.}

We conclude showing  some examples of s-spaces:
\subsection{Lie groups.}

 Let $G$ be a Lie group of dimension $k$. We notated with $e$ the unit of $G$. If $v=\{v_1,\dots,v_n\}$ is a base of $\mathfrak{g}$, let $H_i^v$ be the unique  left invariant vector field on $G$ such that $H_i^{v}(e)=v_i$. Then $\{H_1^v(g),\dots,H_n^v(g)\}$ is base of the tangent space of $G$ at  $g$.

 \begin{exmpl} Given $v$ a basis of $\mathfrak{g}$, let $\lambda^v=(N,\psi,G,R,\{e_i^v\})$ be the s-space over $G$ defined by  $N=G\times G$, $\psi(g,h)=g.h$, $R_a(g,h)=(g.a,a^{-1}.h)$ and $e_i^v(g,h)=H_i^v(g.h)$ for $1\leq i\leq k$. Like $e_i^v \circ
R_a(g,h)=e^v_i(g,h)$, the base change morphism $L^v$ is constantly the identity matrix of $R^{k\times k }$. Therefore, if $T$ is a tensor on $G$ it satisfies that   $$^{\lambda^v}T\circ R_a=^{\lambda^v}T\ .$$
For this reason,  all constant matricial maps come from  a tensor, hence the $\lambda^v-natural$ tensors are in a one to one relation with the matrices of $R^{k\times k}$.

Suppose that $^{\lambda
^{v}}T$ depends only of one parameter, for example  $^{\lambda
^{v}}T(g,h)=^{\lambda ^{v}}T(h)$, then $[^{\lambda^v}T(g',h')]_{ij}=[^{\lambda^v}T(g'hh'^{-1},h')]_{ij}=
 T(g'h)(H_i^v(g'h),H_j^v(g'h))=[^{\lambda^v}T(g',h)]_{ij}
=[^{\lambda^v}T(g,h)]_{ij}$, that is $T$ is  $\lambda^v-natural$. Therefore, $T$ is $\lambda^v-natural$ if and only if $T$ is $^{\lambda
^{v}}T$ depends only of one parameter. The left invariant metrics are tensors of this type.

Let $v'$ be another base of $\mathfrak{g}$ and consider $\lambda^{v'}$. If $a_{vv'}\in GL(k)$ is the matrix that satisfies $v'=a_{vv'}v$, then we have that  $e_i^{v'}(g,h)=e_i^{v}(g,h).a_{vv'}$ and $ ^{\lambda ^{v'}}T=(a_{vv'})^t. ^{\lambda ^{v}}T.a_{vv'}$ for a tensor $T$ on $M$. Thus the set of $\lambda^v-natural$ tensors is independent of the choice of the base $v$. We can observe that $(Id_{G\times G}, Id_{G})$ is a morphism of s-spaces with constant linking map $a_{vv'}$, so $T\in I_{(Id_{G\times G}, Id_{G})}$ if and only if $a_{vv'}\in G_{T}(g,h)$.

\end{exmpl}

\begin{exmpl} Let $\lambda=\{N,\psi, O, R,\{e_i\}\}$ be the s-space over $G$ defined by \linebreak $N=G\times L\mathfrak{g}=\{(g,v):g\in G\ and\ v \mbox{ is a base of } \mathfrak{g}\}$, $\psi(g,v_1,\dots,v_n)=g$, $O=GL(n)$,  $R_{\xi}(g,v)=(g,v.a)$ and $e_i(g,v)=H_i^v(g)$. Since $\{e_i\}\circ R_{\xi}=\{e_i\}.\xi$, $ ^{\lambda}T \circ R_{\xi}=\xi^t.^{\lambda}T.\xi$ for all $\xi \in GL(k)$. Therefore, there is  only one $\lambda-natural$ and is the null tensor.

 The left invariant metrics on $G$ are not $\lambda-naturals$ but for a metric $T$ on $G$  we have that $T$ is a left invariant metric if and only if $^{\lambda}T(g,v)=^{\lambda}T(v)$. If $T$ is a left invariant metric, then  $$[^{\lambda}T(g,v)]_{ij}=T(g)((L_g)_{*_e}(v_i),(L_g)_{*_e}(v_j))=$$
  $$T(e)((L_{g^{-1}})_{*_g}((L_g)_{*_e}(v_i)),(L_{g^{-1}})_{*_g}((L_g)_{*_e}(v_i)))=T(e)(v_i,v_j)=[^{\lambda}T(e,v)]_{ij}.$$ Suppose that the matrix representation induced by $T$ depends only of the parameter of $\mathfrak{g}$. Let $g,h\in G$ and  $w,v\in T_gG$ we have to see that $T(g)(v,w)=T(hg)((L_h)_{*_g}(v),(L_h)_{*_g}(w))$. Let $\{u_1,\dots,u_n\}$ be a base of $\mathfrak{g}$. If
$v=\sum_{i=1}^nv_i(L_g)_{*_e}(u_i)$ and
$w=\sum_{i=1}^nw_i(L_g)_{*_e}(u_i)$, then $(L_h)_{*_g}(v)=\sum_{i=1}^nv_i(L_{hg})_{*_e}(u_i)$ and
$(L_h)_{*_g}(w)=\sum_{i=1}^nw_i(L_{hg})_{*_e}(u_i)$. Hence,
$$T(hg)((L_h)_{*_g}(v),(L_h)_{*_g}(w))=\pmatrix{v_1 &, \dots, &
v_n}.^{\lambda}T(hg,u).\pmatrix{w_1 \cr \vdots \cr w_n}=T(g)(v,w).$$

Let $T$ be a tensor  such that $^{\lambda}T(g,v)$ depends only of $v$. We  know that
$^{\lambda}T(g,v.\xi)=(\xi)^t.^{\lambda}T(e,v).\xi$ for all $\xi\in GL(k)$. Fixed $v_0\in L\mathfrak{g}$ and let $F:L\mathfrak{g}\longrightarrow GL(k)$ be defined by $v=v_0.F(v)$. Then, $^{\lambda}T(g,v)=(F(v))^t.^{\lambda}T(e,v_0).F(v)$ for all $(g,v)\in G\times L\mathfrak{g}$. So we have that

$^{\lambda}T$ depends only of the parameter of $L\mathfrak{g}$ if and only if there exist $A\in  R^{k\times k}$ and a differentiable function $F:L\mathfrak{g}\longrightarrow GL(k)$ that satisfies $F(w.\xi)=F(w).\xi$, such that
$$^{\lambda}T(g,w)=(F(w))^t.A.F(w)$$

\end{exmpl}

\begin{exmpl}
 Fixed $v\in L\mathfrak{g}$ and consider  $\lambda^v=(G\times O(k),\psi, O(k),R ,\{e_i^v\})$ where $\psi(g,\xi)=g$, $R_a(g,\xi)=(g,\xi a)$,  $e_i^v(g,\xi)=H_i^{v.\xi}(g)$ . $\lambda$ is a s-space over $G$ with base change morphism $L=Id_{O(k)}$. If $T$ is a tensor of $M$, then  $^{\lambda}T \circ R_a=a^t.^{\lambda}T.a$. Therefore, $T$ is $\lambda-natural$ if and only if $^{\lambda}T(g,\xi)=f(g).Id_{k\times k}$ with $f:G\longrightarrow\R$ a differentiable function. Is easy to see that $^{\lambda}T((g,\xi).a)=(\xi a)^t.^{\lambda}T (g,Id).(\xi a)$, hence the matrix representation of $T$ depends only of the parameter of $O(k)$ if and only if  $^{\lambda}T(g,\xi)=\xi ^t. A.\xi\ \mbox{ with } A\in \R^{n\times
n}.$

\end{exmpl}

\subsection{Bundle metrics.}

Let $\alpha=(P,\pi, G,\ \cdot\ )$ be a principal fiber bundle  endowed with a connection $\omega$  on a Riema-nnian manifold $(M,g)$. Let us denote with
${\cal {M}}_{{\bf{ad}}}(\mathfrak{g})$ the set of metrics on $\mathfrak{g}$ that are invariant by the adjoint map $\bf{ad}$. Consider the metric on $P$ defined by
\begin{equation}\label{bundlemetric}
 h(p)(X,Y)=g(\pi(p))(\pi_{*_p}(X),\pi_{*_p}(Y))+(l \circ
\pi)(p)(\omega(X),\omega(Y))
\end{equation}
 where $l:M\longrightarrow {\cal
{M}}_{{\bf{ad}}}(\mathfrak{g})$. If $G$ is compact, ${\cal {M}}_{{\bf{ad}}}(\mathfrak{g})\neq \emptyset$, and if $\mathfrak{g}$  is also a semisimple algebra, then essentially there is  (unless scalar multiplication) only one  positive defined  ${\bf{ad}}$-invariant metric \cite{Mil}. If $l$ is a constant function, $h$ is called a \textit{bundle metric}. It is  easy to see that $\pi:(P,h)\longrightarrow (M,g)$ is a Riemannian submersion.

Let $l_0$ be an $\bf{ad}$-invariant map on $\mathfrak{g}$. We are going to consider the s-space $\lambda=(N,\psi,O,R,\{e_i\})$ over $P$ given by $N=\{(q,u,v,g):q\in P,\  u\ \mbox{is an orthonormal base of } M_{\pi(q)},$
$v$ is an orthonormal base of $\mathfrak{g}$ with respect to $l_0$ and $g\in G\}$,    $\
\psi(q,u,v,g)=q.g$, $O=O(n)\times O(k)\times G$ and the action is defined by  $R_{(a,b,h)}(q,u,v,g)=(qh,ua,vb,h^{-1}g)$.
For $1\leq i \leq n$,  $e_i(q,u,v,g)$ is the horizontal lift with respect to $\omega$ of  $u_i$ at $q.g$
  and, for $1\leq j\leq k$, $e_{n+j}(q,u,v,g)$ is the unique vertical vector on $P_{p.g}$ such that $\omega(q.g)(e_{n+j}(q,u,v,g))=v_j$. $\lambda$ is a trivial s-space over $\alpha$.

Let $G$ be a compact Lie group with  $\mathfrak{g}$ a semisimple algebra and $h$ a metric on $P$ of the type of (\ref{bundlemetric}). Then, we have the following proposition:

\begin{prop} $h$ is $\lambda-natural$ with respect to $\alpha$ if and only if $h$ is a bundle metric.

\end{prop}

\begin{proof}{{Proof.}} By definition $^{\lambda}h(q,u,v,g)$ is the matrix of $h(q.g)$ with respect to de base
$\{e_i(q,u,v,g)$ $,e_{n+i}(q,u,v,g)\}$.
For $1\leq i,j \leq n$,  we have that:
$$h(q.g)(e_i(q,u,v,g),e_j(q,u,v,g))=g(u_i,u_j)+0=\delta_{ij}$$
For $1\leq i \leq n$ and $1\leq j \leq k$: $$h(qg)(e_i(q,u,v,g),e_{n+j}(q,u,v,g))=0=h(qg)(e_{n+j}(q,u,v,g),e_i(q,u,v,g))$$
and for $1\leq i,j \leq k$: $$h(q.g)(e_{n+i}(q,u,v,g),e_{n+j}(q,u,v,g))=
l\circ \pi(qg)(v_i,v_j)=f(\pi(q)).\delta_{ij}$$
because $\mathfrak{g}$ has essentially one $\textbf{ad}-invariant\ metric$. Since $$^{\lambda}h(q,u,v,g)=\pmatrix{Id_{n\times n} & 0 \cr 0 &
f(\pi(q)).Id_{k\times k}}$$
$h$ is $\lambda-natural$ with respect to $\alpha$  if and only if $f$ is a constant map, that is to say that $h$  is a bundle metric.
\end{proof}

\begin{rem} If  $\mathfrak{g}$ has different $\textbf{ad}-invariant$ metrics, and $h$ is a metric of the type of (\ref{bundlemetric}),  then $^{\lambda}h:N\longrightarrow\R^{(n+k)\times(n+k)}$ only depends of the parameter of $G$ if $l=\delta.l_0$ with $\delta$ a constant. In general, the metrics of type (\ref{bundlemetric})
that are $\lambda-natural$ with respect to $\alpha$ are the bundle metrics induced by the $\textbf{ad}-invariant\ metric$ $l_0$.
\end{rem}

\begin{rem} The s-space $\lambda$ depends of  the metric $l_0$ and of the connection $\omega$. Let $\omega'$ be another connection on $\alpha$ and consider the s-space $\lambda'$ induced by it.  The difference between the connection are the horizontal subspaces that each one determine and the difference between $\lambda^{\omega}$ and $\lambda^{\omega'}$ are the maps $e_i:N\longrightarrow TP$ and $e_i':N\longrightarrow TP$. Let
$A(p,u,v,g)=\pmatrix{a_1(p,u,v,g) &  a_2(p,u,v,g) \cr
a_4(p,u,v,g) & a_3(p,u,v,g)}\in GL(n+k)$ be the matricial map that satisfies
$\{e_i',e_{n+j}'\}=\{e_i,e_{n+j}\}.A$ where $a_1(p,u,v,g)\in R^{n\times n}$, $a_2(p,u,v,g)\in
\R^{n\times k}$, $a_3(p,u,v,g)\in \R^{k\times k}$ and $a_4(p,u,v,g)\in \R^{k \times n}$. Since $e_{n+j}(p,u,v,g)=e'_{n+j}(p,u,v,g)$, we have that $a_2\equiv 0$ and $a_3\equiv Id_{k\times k}$. If $T$ is a tensor, then

$^{\lambda^{\omega'}}T(p,u,v,g)=\pmatrix{a_1^t(p,u,v,g) &
a_4^t(p,u,v,g) \cr 0 &
Id_{k\times k}}.^{\lambda^{\omega}}T(p,u,v,g).\pmatrix{a_1(p,u,v,g) & 0 \cr
a_4(p,u,v,g) & Id_{k\times k}}$

Suppose as in the proposition above that there is essentially one $\textbf{ad}-invariant $ metric. Then if $h$ is a  metric of type (
\ref{bundlemetric}) we have that

$^{\lambda^{\omega'}}h(p,u,v,g)=$
$$\pmatrix{a_1^t(p,u,v,g)a_1(p,u,v,g)+
f(\pi(p))a_4^t(p,u,v,g).a_4(p,u,v,g) &
f(\pi(p)).a^t_4(p,u,v,g) \cr
f(\pi(p))a_4(p,u,v,g) &
f(\pi(p)).Id_{k\times k}}$$

Therefore, if the connections satisfy that $a_1\in O(n)$ and $a_4$ is a constant map, then $h$ is $\lambda-natural$ with respect to $\alpha$ if and only if $h$ is $\lambda'-natural$ with respect to $\alpha$.  In this situation $h$ is a bundle metric.
\end{rem}

\vspace{0.5 cm}



\vspace{1 cm}

\noindent{\sc Guillermo Henry  \\
Departamento de Matem\'atica,  FCEyN, Universidad de Buenos Aires\\
 Ciudad Universitaria, Pabell\'on I,  Buenos Aires, C1428EHA, Argentina  }\\
{\it e-mail address}: ghenry@dm.uba.ar  \\

\end{document}